\documentclass[a4paper]
{amsart}

\usepackage{fullpage}

\usepackage[all,cmtip]{xy}
\usepackage{amssymb,amsthm}
\usepackage[usenames,dvipsnames]{xcolor}
\usepackage[utf8]{inputenc}
\usepackage{tikz}
\usetikzlibrary{matrix,arrows,calc,fit,positioning,decorations.pathmorphing}
\usepackage{graphicx}
\usepackage{mathrsfs}  
\usepackage{todonotes}
\usepackage{multirow}

\newtheorem{theorem}[equation]{Theorem}
\newtheorem{corollary}[equation]{Corollary}
\newtheorem{lemma}[equation]{Lemma}
\newtheorem{proposition}[equation]{Proposition}
\theoremstyle{definition}
\newtheorem{definition}[equation]{Definition}
\theoremstyle{remark}
\newtheorem{remark}[equation]{Remark}
\newtheorem{example}[equation]{Example}

\numberwithin{equation}{section}

\newcommand{\Ninf}{$N_{\infty}$}

\newcommand{\e}{1}
\newcommand{\D}{\mathcal{N}}
\newcommand{\CC}{\mathcal{C}}
\newcommand{\B}{\mathcal{B}}

\newcommand{\Comp}{\operatorname{Comp}}

\begin{document}

\title
{Equivariant homotopy commutativity for $G=C_{pqr}$}%
\author{Scott Balchin}%
\address{Mathematics Institute, Zeeman Building, University of Warwick, Coventry, CV4 7AL, UK}
\email{scott.balchin@warwick.ac.uk}

\author{Daniel Bearup}%
\address{University of Kent, School of Mathematics, Statistics and Actuarial Science, Sibson Building, Canterbury, CT2 7FS, UK}
\email{d.bearup@kent.ac.uk}

\author{Clelia Pech}%
\address{University of Kent, School of Mathematics, Statistics and Actuarial Science, Sibson Building, Canterbury, CT2 7FS, UK}
\email{c.m.a.pech@kent.ac.uk}

\author{Constanze Roitzheim}%
\address{University of Kent, School of Mathematics, Statistics and Actuarial Science, Sibson Building, Canterbury, CT2 7FS, UK}
\email{csrr@kent.ac.uk}

\subjclass
{}

\begin{abstract}
We investigate the combinatorial data arising from the classification of equivariant homotopy commutativity for cyclic groups of order $G=C_{p_1 \cdots p_n}$ for $p_i$ distinct primes. In particular, we will prove a structural result which allows us to enumerate the number of $N_\infty$-operads for $C_{pqr}$, verifying a computational result.
\end{abstract}

\maketitle

\tableofcontents



\section{Introduction}

A (symmetric topological) operad is a sequence of spaces $\mathcal{O}(n)$ for $n\geq 0$, equipped with an action of the symmetric group $\Sigma_n$ and compatibility conditions, where $\mathcal{O}(n)$ encodes the possibilities of $n$-ary operations. An $\mathcal{O}$-algebra is then a topological space $X$ together with maps $$\mathcal{O}(n) \times_{\Sigma_n} X^n \longrightarrow X$$ plus compatibility conditions. If $\mathcal{O}(n)$ is $\Sigma_n$-contractible for each $n$, we speak of an \emph{$E_\infty$-operad.} This type of operad governs homotopy commutativity, as the contractibility implies that all the different choices of multiplying $n$ elements are homotopic. There are many different $E_\infty$-operads which all have their own technical advantages, but as the homotopy theory of operads depends on the homotopy type of the underlying spaces, all $E_\infty$-operads are equivalent in this sense, meaning that there is one notion of homotopy commutativity.

Equivariantly, this is a different story. If we move on from spaces to $G$-spaces for a finite group $G$, we do not just have $\Sigma_n$ acting on $X^n$ in the usual way but we also have to consider $G$ permuting the factors of any product indexed over $G$-sets. This $G$-action needs to be compatible with the $\Sigma_n$-action. This then leads to the notion of \emph{$N_\infty$-operads.} Unlike in the nonequivariant case, not all $N_\infty$-operads are weakly equivalent to each other.  Instead, those equivalence classes are determined by so-called \emph{transfer systems}, which are combinatorial data consisting of pairs of subgroups of $G$ satisfying some conditions. Conversely, every such transfer system also determines an $N_\infty$-operad. In particular, transfer systems can be depicted as graphs satisfying certain conditions, which we call an \emph{$N_\infty$-diagram}.

It then becomes an intriguing question to see how many different types of equivariant homotopy commutativity are possible for a finite group $G$, and how these are related. For a cyclic group of order $p^{n-1}$, an answer was given in \cite{BBR}, namely, the number of $N_\infty$-diagrams for $C_{p^{n-1}}$ is the $n^{th}$ Catalan number. Moreover, the set of $N_\infty$-diagrams for a fixed group is a lattice, which in the case of $G=C_{p^{n-1}}$ is isomorphic to the $n$-Tamari lattice (the vertex set of the $n$-associahedron). Therefore, types of equivariant homotopy commutativity have interesting links with well-studied combinatorial objects. 


Having thus covered cyclic groups of order equal to a prime power, one might be tempted to think that there would be a similarly neat answer for any finite abelian group. However, one would soon find out that this is not the case as strange ``mixed'' diagrams appear whenever one considers products of groups. For example, for $C_p$ there are two possible $N_\infty$-diagrams, for $C_{pq}$, $p \neq q$, there are ten, and we will show that for $C_{pqr}$ for distinct primes $p$, $q$ and $r$ there are 450. We will also explain why the number grows very rapidly for $C_{p_1 \cdots p_n}$ and present some structural insights into the general case. 

Below we outline the main result of the paper, which gives a structural result on the collection of $N_\infty$-operads, and suggests a more economical method to compute them.

\begin{theorem}
The set of $N_\infty$-diagrams for $G = C_{p_1 \cdots p_n}$ admits a decomposition into $(n+1)$ disjoint subsets
$$\mathcal{N}_n = \bigsqcup^n_{d=0} \operatorname{Comp}_d(G).$$
Moreover, there is an involution $\Phi_n$ on $\D_n$ such that
$$\Phi_n (\operatorname{Comp}_d(G)) = \operatorname{Comp}_{n-d}(G)$$
for any $0 \leq d \leq n$. In particular, we have
$$| \operatorname{Comp}_d(G) | = | \operatorname{Comp}_{n-d}(G) |.$$
\end{theorem}
%
%


\section{$N_\infty$-operads and $N_\infty$-diagrams}

Given a topological space $X$ equipped with a multiplication $m: X \times X \longrightarrow X$, we would like to say that its multiplication is \emph{homotopy commutative} if the diagram
\[
\xymatrix{ X \times X \ar[d]_-{\tau} \ar[rr]^-m  & & X \\
X \times X \ar[rru]_-m & &
}
\]
commutes up to homotopy, where $\tau$ is the twist map permuting the two factors. With this in place one would now have to take care of coherence, that is, the chosen homotopy between $m$ and $m \circ \tau$ needs to be compatible with multiplying three or more copies of $X$. Such coherence issues are neatly packaged in the theory of operads~\cite{loop}.

\begin{definition}
A (topological) \emph{symmetric operad} is a collection $\mathcal{O}=\{\mathcal{O}(n)\}_{n \ge 0}$ of topological spaces $\mathcal{O}(n)$ equipped with a (right) $\Sigma_n$-action together with maps
\[
\mathcal{O}(n) \times \mathcal{O}(i_1) \times \cdots \times \mathcal{O}(i_n) \longrightarrow \mathcal{O}(i_1 + \cdots + i_n)
\]
such that the expected coherence diagrams hold with regards to associativity, unitality and the symmetric group actions.
\end{definition}

An \emph{algebra over an operad} $\mathcal{O}$ is a space $X$ together with multiplication maps
\[
\mathcal{O}(n) \times_{\Sigma_n} X^n \longrightarrow X
\]
satisfying the expected coherence diagrams. Here, $\Sigma_n$ acts on $X^n$ by permuting factors. This information is equivalent to a morphism of symmetric operads
\[
\mathcal{O} \longrightarrow \operatorname{End}(X),
\]
where $\operatorname{End}(X)$ denotes the endomorphism operad of $X$, i.e., 
\[
\operatorname{End}(X)(n)=\operatorname{Hom}(X^n, X).
\]
We can think of the space $\mathcal{O}(n)$ as the different possibilities of multiplying $n$ elements in our space. For instance, if $\mathcal{O}(n)=\ast$ for all $n$, then there is a unique way of multiplying $n$ elements
\[
\ast \times_{\Sigma_n} X^n \cong X^n/\Sigma_n \longrightarrow X.
\]
In particular, this means that our space $X$ is a strictly commutative object. 

If instead we suppose that $\mathcal{O}(n)\simeq \ast$ for all $n$, then there is one way of multiplying $n$ elements ``up to homotopy'', which leads to the following definition.

\begin{definition}
An \emph{$E_\infty$-operad} is a symmetric operad $\mathcal{O}$ such that the action of $\Sigma_n$ on each space is free, and every $\mathcal{O}(n)$ is $\Sigma_n$-equivariantly contractible. 
\end{definition}

There are many different $E_\infty$-operads, each of them having their own technical advantages and disadvantages. Thankfully, all $E_\infty$-operads are weakly equivalent, indeed, there is a Quillen model structure on the category of topological symmetric operads where the weak equivalences are those maps that are levelwise homotopy equivalences of spaces~\cite{BM}. In particular, we can think of this as having one unique (up to homotopy) notion of homotopy commutativity.

Now that we have outlined the theory of homotopy commutativity in the non-equivariant case, we move towards to the more complex setting of $G$-spaces for $G$ some finite group. We now need to consider multiplication maps of the form
\[
\prod\limits_{T} X \longrightarrow X
\]
where $T$ is a $G$-set with $n=|T|$ elements. The $G$-action induces a group homomorphism $G \to \Sigma_n$. This means that the $\mathcal{O}(n)$ spaces should not be thought of merely as  $\Sigma_n$-spaces, but as $(G \times \Sigma_n)$-spaces. Note that simply putting a trivial $G$-action on the $\mathcal{O}(n)$ would not allow for multiplications of the above kind for any $T$ with more than one element. 

We shall now work towards the theory of $N_\infty$-operads, which allows us to fix this issue.

\begin{definition}
A \emph{graph subgroup} $\Gamma$ of $G \times \Sigma_n$ is a subgroup such that $\Gamma \cap (1 \times \Sigma_n)$ is trivial. (Here $1$ denotes the trivial group.)
\end{definition}

Any graph subgroup is of the form
\[
\Gamma = \{ (h, \sigma(h)) \,\,| \,\, h \in H \},
\]
with $H \leq G$ and $\sigma \colon H \longrightarrow \Sigma_n$ a group homomorphism. Moreover, given a finite $H$-set $T$ with $n$ elements we obtain a graph subgroup
\[
\Gamma(T)=\{ (h, \sigma(h)) \,\,|\,\, h \in H \},
\]
where $\sigma \colon H \longrightarrow \Sigma_n$ represents the $H$-action on $T$. Conversely, we can view any graph subgroup as one of the form $\Gamma(T)$, as for 
\[
\Gamma = \{ (h, \sigma(h)) \,\,| \,\, h \in H \},
\]
we can set $T$ to be a set of $n$ elements with the $H$-action given by $\sigma$. 

%

\begin{definition}
An \emph{$N_\infty$-operad} is a symmetric operad $\mathcal{O}$ in the category of $G$-spaces (that is, a collection of $G \times \Sigma_n$-spaces $\mathcal{O}(n), n \ge 0$) satisfying the following conditions.
\begin{itemize}
\item For all $n \ge 0$, $\mathcal{O}(n)$ is $\Sigma_n$-free.
\item For every graph subgroup $\Gamma$ of $G \times \Sigma_n$, the space $\mathcal{O}(n)^\Gamma$ is either empty or contractible.
\item $\mathcal{O}(0)^G$ and $\mathcal{O}(2)^G$ are both nonempty.
\end{itemize}
\end{definition}

The last condition ensures that the operad possesses an equivariant multiplication and an equivariant `point'. 

The second point together with the operad structure implies that each $\mathcal{O}(n)$ is a classifying space for a family of subgroups which satisfy some further properties forced by operad structure. This information can be distilled into the theorem below.

\begin{theorem}
Up to weak equivalence, every $N_\infty$-operad determines and is determined by a set $X=\{ N_K^H \}$, where $K<H$ are subgroups of $G$, satisfying the following properties and their conjugacies.
\begin{itemize}
\item (Transitivity) If $N_K^H \in X$ and $N_H^L\in X$, then $N_K^L \in X$.
\item (Restriction) If $N_K^H \in X$ and $L \leq G$, then $N_{K \cap L}^{H \cap L} \in X$.
\end{itemize}
We will call such a set a \emph{transfer system} and the objects $N_K^H$ will sometimes be referred to as \emph{norm maps}.
\end{theorem}

Blumberg and Hill showed that every operad determines an ``indexing system''~\cite{BH}. Rubin~\cite{Rubin}, Gutierrez-White~\cite{GW} and Bonventre-Pereira~\cite{BP} independently showed that for every such indexing system one can construct a corresponding operad. Barnes-Balchin-Roitzheim~\cite{BBR} showed that indexing systems are equivalent to the transfer systems given in the above version of this theorem.

\begin{corollary}
There are as many homotopy types of $N_\infty$-operads for a fixed finite group $G$ as there are transfer systems for $G$. \qed
\end{corollary}

In particular, there can be only finitely many $N_\infty$-operads for a finite group $G$, and as such, it makes sense to count them. We will denote by $N_\infty(G)$ the set of all $N_\infty$-operads on $G$. For $G=C_{p^n}$, the number of $N_\infty$-operads plus some additional structure has been determined in~\cite{BBR}. 

Before continuing, let us assess the first non-trivial case.  We will choose to display indexing systems as graphs whose vertices are the subgroups of $G$, and there is an edge $H \to K$ if $N_H^K \in X$. 

\begin{example}
Let $G=C_p$ for some prime $p$, then there are two $N_\infty$-operads which have the following graph representations.
\begin{figure}[h!]
 \begin{tikzpicture}[->, node distance=2cm, auto]
\node (1) at (6.000000,0) {$C_{p^1}$};
\node (2) at (4.000000,0) {$C_{p^0}$};
  \node at (3.5,0) {\Huge{(}};
\node at (6.5,0) {\Huge{)}};
\node (11) at (10.000000,0) {$C_{p^1}$};
\node (22) at (8.000000,0) {$C_{p^0}$};
  \node at (7.5,0) {\Huge{(}};
\node at (10.5,0) {\Huge{)}};
\draw (22) to (11);
 \end{tikzpicture}
\end{figure}
\end{example}

The key ingredient in the result of~\cite{BBR} is an operation
\[
\odot \colon N_\infty(C_{p^i}) \times N_\infty(C_{p^j}) \to N_\infty(C_{p^{i+j+2}}).
\]

In particular it is proved that every $N_\infty$-operad for $C_{p^{i+j+2}}$ is of the form $X \odot Y$ for $X \in N_\infty(C_{p^i})$ and $Y \in N_\infty(C_{p^j})$. This then allows an inductive strategy of proof of the main result.

\begin{theorem}[{\cite[Theorem 1]{BBR}}]
For $n \geq 1$ we have
\[
|N_\infty(C_{p^{n}})| = \mathsf{Cat}(n+1)
\]
where $\mathsf{Cat}(n)$ is the $n^{th}$ Catalan number.
\end{theorem}

The result goes a bit further than just an enumeration result. We can put a partial order on the set of all $N_\infty$-diagrams for a fixed group by saying that one $N_\infty$-operad $X$ is smaller than another $N_\infty$-operad $Y$ if it is a subset of $Y$. On the other side, there is a wealth of objects enumerated by Catalan numbers. One of them is the set of rooted binary trees. We can put a partial order on the set of rooted binary trees with $n$ leaves by saying that one tree is larger than another if it can be obtained from the latter by rotating a branch to the right. Balchin-Barnes-Roitzheim found that, indeed, $N_\infty$-diagrams for $G=C_{p^n}$ and rooted binary trees with $n+2$ leaves are isomorphic as posets~\cite{BBR}. However, we do not wish to elaborate on this result here.

The goal of this paper is to study the set $N_\infty(G)$ for $G$ a group of the form $C_{p_1 \cdots p_n}$ for $p_i$ distinct primes where the situation is somewhat more complicated. Note that, in particular, the subgroup lattice is an $n$-dimensional cube. 

\section{Classifying \Ninf-operads for $G=C_{p_1p_2}$}\label{s:2d}

In this section we explore the structure of $N_\infty$-operads for $G=C_{p_1p_2}$ as it will illuminate the theory that we present in the rest of the paper. Specifically, it is the only non-trivial case where one can visualise the entire situation. One can check that there are ten such $N_\infty$-operads as follows.

\begin{figure}[h!]
 \begin{tikzpicture}[->, node distance=2cm, auto,scale = 0.6]
\node (11) at (2.000000,0) {$C_{p_1}$};
\node (p1) at (0.000000,0) {$\e$};
\node (q1) at (2.000000,-2.0) {$C_{p_1p_2}$};
\node (pq1) at (0.000000,-2.0) {$C_{p_2}$};

\node (12) at (7.000000,0) {$C_{p_1}$};
\node (p2) at (5.000000,0) {$\e$};
\node (q2) at (7.000000,-2.0) {$C_{p_1p_2}$};
\node (pq2) at (5.000000,-2.0) {$C_{p_2}$};
\draw (p2) to (pq2);

\node (13) at (12.000000,0) {$C_{p_1}$};
\node (p3) at (10.000000,0) {$\e$};
\node (q3) at (12.000000,-2.0) {$C_{p_1p_2}$};
\node (pq3) at (10.000000,-2.0) {$C_{p_2}$};
\draw (p3) to (13);

\node (14) at (17.000000,-0) {$C_{p_1}$};
\node (p4) at (15.000000,-0) {$\e$};
\node (q4) at (17.000000,-2.0) {$C_{p_1p_2}$};
\node (pq4) at (15.000000,-2.0) {$C_{p_2}$};
\draw (p4) to (pq4);
\draw (p4) to (14);

\node (15) at (22.000000,-0) {$C_{p_1}$};
\node (p5) at (20.000000,-0) {$\e$};
\node (q5) at (22.000000,-2) {$C_{p_1p_2}$};
\node (pq5) at (20.000000,-2) {$C_{p_2}$};
\draw (15) to (q5);
\draw (p5) to (pq5);

\node (16) at (2.000000,-5) {$C_{p_1}$};
\node (p6) at (0.000000,-5) {$\e$};
\node (q6) at (2.000000,-7) {$C_{p_1p_2}$};
\node (pq6) at (0.000000,-7) {$C_{p_2}$};
\draw (p6) to (16);
\draw (p6) to (q6);
\draw (p6) to (pq6);
\draw (pq6) to (q6);
\draw (16) to (q6);

\node (17) at (7.000000,-5) {$C_{p_1}$};
\node (p7) at (5.000000,-5) {$\e$};
\node (q7) at (7.000000,-7.0) {$C_{p_1p_2}$};
\node (pq7) at (5.000000,-7.0) {$C_{p_2}$};
\draw (p7) to (q7);
\draw (p7) to (pq7);
\draw (pq7) to (q7);
\draw (p7) to (17);

\node (18) at (12.000000,-5) {$C_{p_1}$};
\node (p8) at (10.000000,-5) {$\e$};
\node (q8) at (12.000000,-7.0) {$C_{p_1p_2}$};
\node (pq8) at (10.000000,-7.0) {$C_{p_2}$};
\draw (p8) to (18);
\draw (p8) to (q8);
\draw (18) to (q8);
\draw (p8) to (pq8);

\node (19) at (17.000000,-5) {$C_{p_1}$};
\node (p9) at (15.000000,-5) {$\e$};
\node (q9) at (17.000000,-7) {$C_{p_1p_2}$};
\node (pq9) at (15.000000,-7) {$C_{p_2}$};
\draw (p9) to (19);
\draw (p9) to (q9);
\draw (p9) to (pq9);

\node (110) at (22.000000,-5) {$C_{p_1}$};
\node (p10) at (20.000000,-5) {$\e$};
\node (q10) at (22.000000,-7) {$C_{p_1p_2}$};
\node (pq10) at (20.000000,-7) {$C_{p_2}$};
\draw (p10) to (110);
\draw (pq10) to (q10);
\end{tikzpicture}
\caption{The ten possible $N_\infty$-operad structures for $G=C_{p_1p_2}$.}
\label{fig-nops-cpq}
\end{figure}
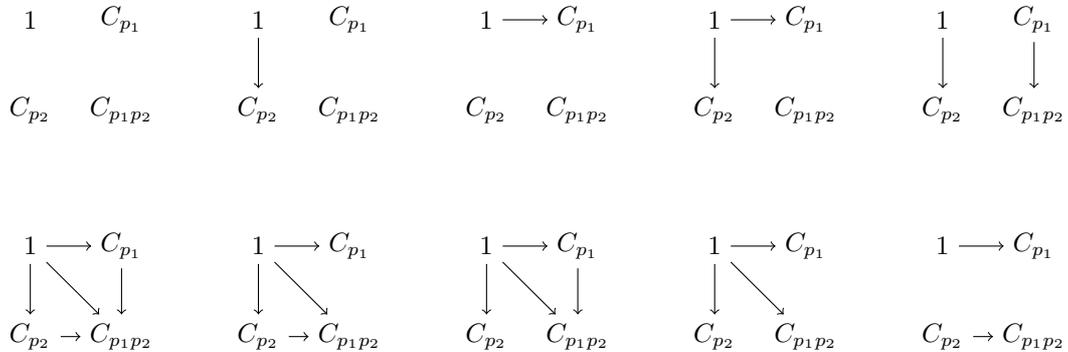

Note that there is an odd one out in these diagrams, namely the following diagram, which has a diagonal which is not forced by the restriction rule. It is this type of $N_\infty$-operad, which we call a ``mixed diagram'', that causes the complexity in this problem.

\begin{figure}[h!]
 \begin{tikzpicture}[->, node distance=2cm, auto,scale = 0.75]
\node (11) at (2.000000,0) {$C_{p_1}$};
\node (p1) at (0.000000,0) {$\e$};
\node (q1) at (2.000000,-2.0) {$C_{p_1p_2}$};
\node (pq1) at (0.000000,-2.0) {$C_{p_2}$};
\draw (p1) to (11);
\draw (p1) to (q1);
\draw (p1) to (pq1);
\end{tikzpicture}
\caption{The \emph{mixed} $N_\infty$-operad for $G=C_{p_1 p_2}$.}
\end{figure}
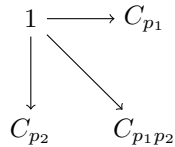

Now, we will explore the structure of this collection of ten operads, which we denote by $\mathcal{N}_2$. Let us consider three subsets of $\mathcal{N}_2$. First, denote by $\operatorname{Comp}_0(C_{p_1p_2})$ those $N_\infty$-operads which do not contain the norm map $N_1^{C_{p_1 p_2}}$, i.e. the diagonals in Fig. \ref{fig-nops-cpq}.

\begin{figure}[ht!]
\begin{tikzpicture}[->, node distance=2cm, auto,scale = 0.6]
\node (16) at (2.000000,-5) {$C_{p_1}$};
\node (p6) at (0.000000,-5) {$\e$};
\node (q6) at (2.000000,-7) {$C_{p_1p_2}$};
\node (pq6) at (0.000000,-7) {$C_{p_2}$};
\draw (p6) to (16);
\draw (p6) to (q6);
\draw (p6) to (pq6);
\draw (pq6) to (q6);
\draw (16) to (q6);

\node (17) at (7.000000,-5) {$C_{p_1}$};
\node (p7) at (5.000000,-5) {$\e$};
\node (q7) at (7.000000,-7.0) {$C_{p_1p_2}$};
\node (pq7) at (5.000000,-7.0) {$C_{p_2}$};
\draw (p7) to (q7);
\draw (p7) to (pq7);
\draw (pq7) to (q7);
\draw (p7) to (17);

\node (18) at (12.000000,-5) {$C_{p_1}$};
\node (p8) at (10.000000,-5) {$\e$};
\node (q8) at (12.000000,-7.0) {$C_{p_1p_2}$};
\node (pq8) at (10.000000,-7.0) {$C_{p_2}$};
\draw (p8) to (18);
\draw (p8) to (q8);
\draw (18) to (q8);
\draw (p8) to (pq8);

\node (19) at (17.000000,-5) {$C_{p_1}$};
\node (p9) at (15.000000,-5) {$\e$};
\node (q9) at (17.000000,-7) {$C_{p_1p_2}$};
\node (pq9) at (15.000000,-7) {$C_{p_2}$};
\draw (p9) to (19);
\draw (p9) to (q9);
\draw (p9) to (pq9);
\end{tikzpicture}
\caption{The collection $\operatorname{Comp}_0(C_{p_1p_2})$.}
\end{figure}

Next, we consider the collection $\operatorname{Comp}_2(C_{p_1p_2})$ of those $N_\infty$-operads which do not contain the norm maps $N_{H}^{C_{p_1 p_2}}$ for $H$ any of the two proper subgroups of $C_{p_1 p_2}$.

\begin{figure}[ht!]
\begin{tikzpicture}[->, node distance=2cm, auto,scale = 0.6]
\node (11) at (2.000000,0) {$C_{p_1}$};
\node (p1) at (0.000000,0) {$\e$};
\node (q1) at (2.000000,-2.0) {$C_{p_1p_2}$};
\node (pq1) at (0.000000,-2.0) {$C_{p_2}$};

\node (12) at (7.000000,0) {$C_{p_1}$};
\node (p2) at (5.000000,0) {$\e$};
\node (q2) at (7.000000,-2.0) {$C_{p_1p_2}$};
\node (pq2) at (5.000000,-2.0) {$C_{p_2}$};
\draw (p2) to (pq2);

\node (13) at (12.000000,0) {$C_{p_1}$};
\node (p3) at (10.000000,0) {$\e$};
\node (q3) at (12.000000,-2.0) {$C_{p_1p_2}$};
\node (pq3) at (10.000000,-2.0) {$C_{p_2}$};
\draw (p3) to (13);

\node (14) at (17.000000,-0) {$C_{p_1}$};
\node (p4) at (15.000000,-0) {$\e$};
\node (q4) at (17.000000,-2.0) {$C_{p_1p_2}$};
\node (pq4) at (15.000000,-2.0) {$C_{p_2}$};
\draw (p4) to (pq4);
\draw (p4) to (14);
\end{tikzpicture}
\caption{The collection $\operatorname{Comp}_2(C_{p_1p_2})$.}
\end{figure}

We then define $\operatorname{Comp}_1(C_{p_1p_2})$ to consist of those $N_\infty$-operads which have the norm map  $N_{C_{p_1}}^{C_{p_1 p_2}}$ or $N_{C_{p_2}}^{C_{p_1 p_2}}$ but not $N_{e}^{C_{p_1 p_2}}$. 

\begin{figure}[ht!]
\begin{tikzpicture}[->, node distance=2cm, auto,scale = 0.6]
\node (17) at (7.000000,-5) {$C_{p_1}$};
\node (p7) at (5.000000,-5) {$\e$};
\node (q7) at (7.000000,-7.0) {$C_{p_1p_2}$};
\node (pq7) at (5.000000,-7.0) {$C_{p_2}$};
\draw(p7) to (pq7);
\draw(17) to (q7);

\node (18) at (12.000000,-5) {$C_{p_1}$};
\node (p8) at (10.000000,-5) {$\e$};
\node (q8) at (12.000000,-7.0) {$C_{p_1p_2}$};
\node (pq8) at (10.000000,-7.0) {$C_{p_2}$};
\draw(p8) to (18);
\draw(pq8) to (q8);
\end{tikzpicture}
\caption{The collection $\operatorname{Comp}_1(C_{p_1p_2})$.}
\end{figure}
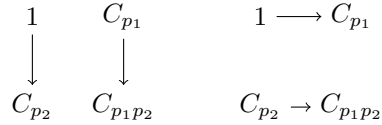

We can see, purely by inspection, that these three ($=2+1$) subsets form a partition of $\mathcal{N}_2$. This will be the first part of the general strategy for understanding $\mathcal{N}_n$. We will prove that we can partition the set $\mathcal{N}_n$ into $(n+1)$ disjoint subsets.

However, there is still some extra structure on the collection $\{ \operatorname{Comp}_i(C_{p_1p_2}) \}_i$, which we will now investigate. Indeed, it is no coincidence that $|\operatorname{Comp}_0(C_{p_1p_2})| = |\operatorname{Comp}_2(C_{p_1p_2})|$. In particular, there exists an involution $\Phi_2 \colon \mathcal{N}_2 \to \mathcal{N}_2$. Instead of giving a formal definition (which will appear in the following section) we simply illustrate it by giving the correspondence and inviting the reader to understand the relationship. We have grouped the elements in coloured blocks to distinguish the $\operatorname{Comp}_i(C_{p_1p_2})$. Note that $\Phi_2 \colon \operatorname{Comp}_i(C_{p_1p_2}) \to \operatorname{Comp}_{2-i}(C_{p_1p_2})$. Also observe that the long diagonal is affected only in the case where $i=0,2$.

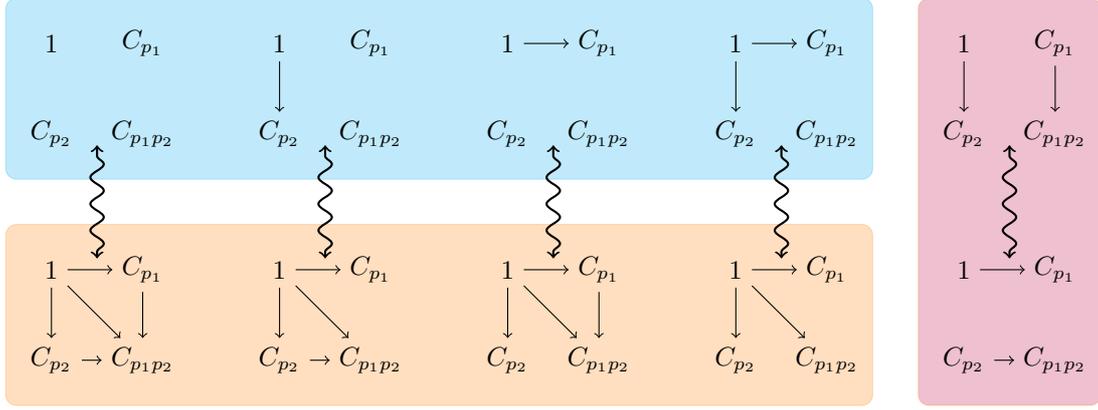
\begin{figure}[h!]
 \begin{tikzpicture}[->, node distance=2cm, auto,scale = 0.6]
\draw[rounded corners, cyan, fill=cyan, opacity=0.25] (-1, -3) rectangle (18, 1) {};
\draw[rounded corners, orange, fill=orange, opacity=0.25] (-1, -8) rectangle (18,-4) {};
\draw[rounded corners, orange, fill=purple, opacity=0.25] (19, -8) rectangle (23,1) {};

\draw[<->, decorate, decoration={snake,pre length=4pt,post length=2pt},thick] (1,-2.25) -- (1,-4.75);
\draw[<->, decorate, decoration={snake,pre length=4pt,post length=2pt},thick] (6,-2.25) -- (6,-4.75);
\draw[<->, decorate, decoration={snake,pre length=4pt,post length=2pt},thick] (11,-2.25) -- (11,-4.75);
\draw[<->, decorate, decoration={snake,pre length=4pt,post length=2pt},thick] (16,-2.25) -- (16,-4.75);
\draw[<->, decorate, decoration={snake,pre length=4pt,post length=2pt},thick] (21,-2.25) -- (21,-4.75);

\node (11) at (2.000000,0) {$C_{p_1}$};
\node (p1) at (0.000000,0) {$1$};
\node (q1) at (2.000000,-2.0) {$C_{p_1p_2}$};
\node (pq1) at (0.000000,-2.0) {$C_{p_2}$};

\node (12) at (7.000000,0) {$C_{p_1}$};
\node (p2) at (5.000000,0) {$1$};
\node (q2) at (7.000000,-2.0) {$C_{p_1p_2}$};
\node (pq2) at (5.000000,-2.0) {$C_{p_2}$};
\draw (p2) to (pq2);

\node (13) at (12.000000,0) {$C_{p_1}$};
\node (p3) at (10.000000,0) {$1$};
\node (q3) at (12.000000,-2.0) {$C_{p_1p_2}$};
\node (pq3) at (10.000000,-2.0) {$C_{p_2}$};
\draw (p3) to (13);

\node (14) at (17.000000,-0) {$C_{p_1}$};
\node (p4) at (15.000000,-0) {$1$};
\node (q4) at (17.000000,-2.0) {$C_{p_1p_2}$};
\node (pq4) at (15.000000,-2.0) {$C_{p_2}$};
\draw (p4) to (pq4);
\draw (p4) to (14);

\node (15) at (22.000000,-0) {$C_{p_1}$};
\node (p5) at (20.000000,-0) {$1$};
\node (q5) at (22.000000,-2) {$C_{p_1p_2}$};
\node (pq5) at (20.000000,-2) {$C_{p_2}$};
\draw (15) to (q5);
\draw (p5) to (pq5);

\node (16) at (2.000000,-5) {$C_{p_1}$};
\node (p6) at (0.000000,-5) {$1$};
\node (q6) at (2.000000,-7) {$C_{p_1p_2}$};
\node (pq6) at (0.000000,-7) {$C_{p_2}$};
\draw (p6) to (16);
\draw (p6) to (q6);
\draw (p6) to (pq6);
\draw (pq6) to (q6);
\draw (16) to (q6);

\node (17) at (7.000000,-5) {$C_{p_1}$};
\node (p7) at (5.000000,-5) {$1$};
\node (q7) at (7.000000,-7.0) {$C_{p_1p_2}$};
\node (pq7) at (5.000000,-7.0) {$C_{p_2}$};
\draw (p7) to (q7);
\draw (p7) to (pq7);
\draw (pq7) to (q7);
\draw (p7) to (17);

\node (18) at (12.000000,-5) {$C_{p_1}$};
\node (p8) at (10.000000,-5) {$1$};
\node (q8) at (12.000000,-7.0) {$C_{p_1p_2}$};
\node (pq8) at (10.000000,-7.0) {$C_{p_2}$};
\draw (p8) to (18);
\draw (p8) to (q8);
\draw (18) to (q8);
\draw (p8) to (pq8);

\node (19) at (17.000000,-5) {$C_{p_1}$};
\node (p9) at (15.000000,-5) {$1$};
\node (q9) at (17.000000,-7) {$C_{p_1p_2}$};
\node (pq9) at (15.000000,-7) {$C_{p_2}$};
\draw (p9) to (19);
\draw (p9) to (q9);
\draw (p9) to (pq9);

\node (110) at (22.000000,-5) {$C_{p_1}$};
\node (p10) at (20.000000,-5) {$1$};
\node (q10) at (22.000000,-7) {$C_{p_1p_2}$};
\node (pq10) at (20.000000,-7) {$C_{p_2}$};
\draw (p10) to (110);
\draw (pq10) to (q10);
\end{tikzpicture}
\caption{The involution $\Phi_2 \colon \mathcal{N}_2 \to \mathcal{N}_2$.}
\end{figure}


\section{Higher dimensions}

We now introduce the general theory of $N_\infty$-diagrams for $G=C_{p_1 \cdots p_n}$, where $p_1,\dots,p_n$ are distinct primes.  We will denote by $\D_n$ the set of $N_\infty$-operads for such a $G$. The goal of this section is to prove that there is an intuitive decomposition of $\D_n$ into $(n+1)$ disjoint subsets, which, as hinted at in Section~\ref{s:2d}, admits an involution $\Phi_n \colon \D_n \to \D_n$. We then use this result in Section~\ref{s:3d} to prove that $|\D_3| = 450$.


\subsection{Decomposition}

\begin{definition}
	Let $D$ be an $N_\infty$-diagram for $G$. If $H < K$ are subgroups of $G$ we denote by $D_H^K$ the $N_\infty$-diagram induced by $D$ on the vertices corresponding to subgroups of $K$ containing $H$.
\end{definition}

The set of $N_\infty$-diagrams for $G=C_{p_1 \cdots p_n}$ admits a decomposition into $(n+1)$ disjoint subsets as follows. Let $D \in \D_n$, and consider the set of all arrows $( H \to G )$ contained in $D$. Let $G^0$ be the intersection of all the initial vertices of such arrows; clearly, the induced diagram $D^0 := D_{G^0}^G$ contains all the arrows with final vertex $G$, and it is minimal for this property. Therefore, denoting by $\Comp_d(G)$ the set of diagrams $G$ with $G^0 = C_{p_{i_1} p_{i_2} \dots p_{i_{n-d}}}$ for any $0 \leq d \leq n$, we obtain a decomposition:
\[
	\D_n = \bigsqcup_{d=0}^n \Comp_d(G).
\]

\begin{remark}
Note that if $D$ is in $\Comp_d(G)$ then $D^0$ is supported on (and contains the big diagonal of) a $d$-dimensional face. In particular we see that $\Comp_0(n)$ consists of those $N_\infty$-diagrams which do not contain a norm map to the group itself, and $\Comp_n(G)$ consists of those $N_\infty$-diagrams which contain the long diagonal $N^G_\e$.
\end{remark}


To prove the next result we need to introduce some notation for the facets (codimension one faces) of the $n$-dimensional cube. Any facet has one of the following forms:
\begin{itemize}
	\item \emph{bottom facet} $B_i$, i.e., facet containing the $\e$ vertex and a vertex of the form $C_{p_1 \dots \hat p_i \dots p_n}$ ($\hat p_i$ means removing $p_i$),
	\item \emph{top facet} $T_i$, i.e., facet containing a vertex of the form $C_{p_i}$ and the $G$ vertex.
\end{itemize}

\begin{proposition}\label{p:decomp}
	If $D \in \Comp_d(G)$, then there exist $n-d$ facets adjacent to $\e$ and not intersecting $D^0$ such that all arrows of $D$ are either arrows of $D^0$ or contained in these $n-d$ facets.
\end{proposition}

\begin{proof}
	First notice that if $D \in \Comp_n(G)$, then $D^0=D$ so the result is trivially true. Now assume $d < n$ and $D \in \Comp_d(G)$. Clearly, $D^0$ is then contained in the intersection of $n-d$ top facets, say $T_1,T_2,\dots,T_{n-d}$. This implies that the intersection $G^0$ of all initial vertices of arrows with final vertex $G$ is given by
	\[
		G^0=C_{p_1 p_2 \dots p_{n-d}}.
	\]
	Let us show that $D^0$ must include an arrow $(G^0 \to G)$. To do this let us enumerate the set of arrows of $D^0$ with final vertex $G$:
	\[
		(H_1 \to G), (H_2 \to G), \dots, (H_k \to G).
	\]
	Since $D$ contains the arrows $(H_1 \to G)$ and $(H_2 \to G)$, by the restriction condition it must have an arrow $(H_1 \cap H_2 \to H_1)$. But by the transitivity condition, the arrows $$(H_1 \cap H_2 \to H_1) \,\,\, \mbox{and} \,\,\, (H_1 \to G)$$ imply the existence of an arrow $(H_1 \cap H_2 \to G)$. Repeating this argument, since $G^0= \bigcap_{i=1}^k H_i,$ we deduce that $D^0$ has to contain the arrow $(G^0 \to G)$.
	
	Since $D^0$ contains an arrow with initial vertex $G^0=C_{p_1 p_2 \dots p_{n-d}}$, there are exactly $n-d$ bottom faces which do not intersect $D^0$, namely $$B_1,B_2,\dots,B_{n-d}.$$ We need to show that any arrow of $D$ is either an arrow of $D^0$, or contained in one of the bottom facets $$B_1,B_2,\dots,B_{n-d}.$$.
	
	Assume by contradiction that it is not the case. Note that an arrow is not in $D^0$ if and only if its initial vertex does not contain $G_0$, and an arrow is not in one of the bottom facets $$B_1,B_2,\dots,B_{n-d}$$ if and only if its final vertex is not in the union of subgroups $\bigcup_{i=1}^{n-d} C_{p_1 \dots \hat p_i \dots p_n}$. Therefore, $D$ must have an arrow $(K \to L)$ with $K \not\supset G^0$ and 
	\[
		L \not \subset \bigcup_{i=1}^{n-d} C_{p_1 \dots \hat p_i \dots p_n}.
	\]
	The latter condition implies that $L$ must contain the subgroup $C_{p_1 \dots p_{n-d}}=G^0$. By the restriction condition, since $D$ contains the arrow $(K \to L)$ it must also contain the arrow 
	\[
		(K \cap G^0 \to L\cap G^0=G^0).
	\]
	But we saw that $D$ contains the arrow $(G^0 \to G)$, so by transitivity $D$ must contain the arrow $$(K \cap G^0 \to G).$$ Since $K \not \supset G^0$, the subgroup $K \cap G^0$ is a strict subset of $G^0$ with an arrow to $G$, which contradicts the minimality of $G^0$. Therefore, any arrow of $D$ is either an arrow of $D^0$, or contained in one of the bottom facets $$B_1,B_2,\dots,B_{n-d},$$ which concludes the proof.
\end{proof}

\subsection{An involution of $N_\infty$-diagrams for $G=C_{p_1 \cdots p_n}$}

In this section we introduce an involution $\Phi_n : \D_n \to \D_n$, which  swaps the distinguished subsets $\Comp_0(G)$ and $\Comp_n(G)$. We construct the involution by induction on $n \geq 1$ as follows.
\begin{itemize}
	\item If $n=1$, then we let $\Phi_1$ swap the empty $N_\infty$-diagram and the full $N_\infty$-diagram.
	\item Now assume that we constructed the map $\Phi_n$ for a fixed $n \geq 1$, and consider a $N_\infty$-diagram $D$ on the $(n+1)$-dimensional cube, i.e., for $$G=C_{p_1 \cdots p_{n+1}}.$$
To define $\Phi_{n+1}$, we apply $\Phi_n$ to $D$ restricted to each facet and we reindex the vertices so that in the image a vertex $H$ is replaced with $G/H$. In particular this means that $\Phi_{n+1}$ sends a bottom facet $B_i$ to the top facet $T_i$, and vice-versa. Finally, the big diagonal $(1 \to G)$ belongs to $\Phi_{n+1}(D)$ if and only if $D \in \Comp_{n+1}$.
\end{itemize}
We claim that this construction gives a well-defined map from the set of $N_\infty$-diagrams to the set of graphs on the hypercube. Later we will prove that the image of a $N_\infty$-diagram is also a $N_\infty$-diagram.

\begin{proposition}\label{p:phi-welldefined}
For any $n \geq 1$ the map $\Phi_n$ is a well-defined map from the set $\D_n$ of $N_\infty$-diagrams on the $n$-dimensional hypercube to the set of graphs on the hypercube.
\end{proposition}

\begin{proof}
	First notice that $\Phi_n$ acts as the complement on cube edges, with vertices swapped as follows: 
	\[
	H \mapsto G/H. 
	\]
	Thus, it is well-defined on cube edges, and we only need to check that it is well-defined on all other arrows. If $n=1$ all arrows are cube edges, so there is nothing to check. Now assume $\Phi_m$ is well-defined for $m \leq n$, and consider a $N_\infty$-diagram $D$ on the $(n+1)$-dimensional hypercube. The induction hypothesis shows that $\Phi_{n+1}(D)$ is well-defined on all the diagonals of the form $(H \to K)$ when $H \neq 1$ or $K \neq G$. Indeed, those are diagonals of smaller hypercubes. So we only need to consider the instructions for definining $\Phi_n$ on the big diagonal $(1 \to G)$, which is clear.
\end{proof}

We now prove that the image of an $N_\infty$-diagram is an $N_\infty$-diagram.

\begin{theorem}
	For any $N_\infty$-diagram $D \in \D_n$ we have $\Phi_n(D) \in \D_n$.
\end{theorem}

\begin{proof}
	The case $n=1$ is clear. Now assume $\Phi_m(\D_m) \subseteq \D_m$ for $m \leq n$, and consider a $N_\infty$-diagram $D$ on the $(n+1)$-dimensional hypercube. We need to prove the following two properties:
	\begin{itemize}
		\item (restriction condition) for any arrow $(H \to K)$ in $\Phi_{n+1}(D)$ and any subgroup $L$ of $G$ such that $H \cap L \neq K \cap L$, the arrow $(H \cap L \to K \cap L)$ is also in $\Phi_{n+1}(D)$;
		\item (transitivity condition) for any arrows $(H \to K)$ and $(K \to L)$ in $\Phi_{n+1}(D)$, the arrow $(H \to L)$ is also in $\Phi_{n+1}(D)$.
	\end{itemize}
	
	Let us first check that $\Phi_{n+1}$ preserves the restriction condition. Let $(H \to K)$ be any arrow which is not the big diagonal. Then, by induction on the diagram $D_H^K$, which is defined on a smaller cube, we immediately see that the restriction condition is satisfied for all arrows, except maybe the big diagonal. So we may assume that $D$ is such that the arrow $(1 \to G)$ belongs to $\Phi_{n+1}(D)$, and we need to show that the arrows $(1 \to L)$ for any subgroup $L$ of $G$ also belong to $\Phi_{n+1}(D)$. By definition of $\Phi_{n+1}(D)$ we know that $D$ has no arrow adjacent to the vertex $G$. Let $K:=G/L$, and consider the induced diagram $D_K^G$, which has no arrows to $G$ either. By the induction hypothesis it follows that its image by $\Phi$ contains the `big' diagonal $(1 \to G/K)$, that is, $(1 \to L)$. Therefore, $(1 \to L)$ belongs to $\Phi_{n+1}(D)$ as claimed.
	
	Let us now look at the transitivity condition, i.e., let us check that for any arrow $(H \to K)$ and $(K \to L)$ in $\Phi_{n+1}(D)$, the arrow $(H \to L)$ is also in $\Phi_{n+1}(D)$. By induction on a smaller cube, it is immediate to see that the transitivity condition holds when $H \neq 1$ or $L \neq G$. So we only need to prove transitivity for arrows $(1 \to K)$ and $(K \to G)$. 
	
	First, note that the restriction condition on the (smaller) diagram $D_1^K$ implies that all arrows $(1 \to H)$ are in $\Phi_{n+1}(D)$ if $H \subset K$. So let us consider a subgroup $H \neq G$ such that $H \cap K \neq H$. By the restriction condition, since the arrow $(K \to G)$ is in $\Phi_{n+1}(D)$, so is the arrow $(H \cap K \to H)$. Since $H \cap K \subset K$, the arrow $(1 \to H \cap K)$ is in $\Phi_{n+1}(D)$. By transitivity in the diagram $D_1^H$ it follows that the arrow $(1 \to H)$ is in $\Phi_{n+1}(D)$. So we have proved that $\Phi_{n+1}(D)$ contains all arrows $(1 \to H)$, except maybe if $H=G$. 
	
	Now we are left with checking that $\Phi_{n+1}(D)$ contains the big diagonal. Indeed, if it did not, then by definition of $\Phi_{n+1}$ it would mean that $D$ has an arrow $(L \to G)$ for some subgroup $L$. This would imply that $\Phi_{n+1}(D)$ has no arrow $(1 \to G/L)$, which contradicts the fact $\Phi_{n+1}(D)$ must contain all arrows $(1 \to H)$ for $H \neq G$. So $\Phi_{n+1}(D)$ contains the big diagonal, which concludes the proof that $\Phi_{n+1}$ preserves the transitivity condition.
\end{proof}

\begin{proposition}\label{p:enum}\leavevmode
	\begin{enumerate}
		\item The map $\Phi_n$ interchanges the subsets $\Comp_0(G)$ and $\Comp_n(G)$.
		\item The map $\Phi_n$ is an involution.
		\item  $\Phi_n(\Comp_d(G))=\Comp_{n-d}(G)$ for any $0 \leq d \leq n$.
	\end{enumerate}
\end{proposition}

\begin{proof}\leavevmode
	\begin{enumerate}
		\item This is immediate from the construction of $\Phi_n$.
		\item Recall from the proof of Proposition~\ref{p:phi-welldefined} that $\Phi_n$ acts as the complement on cube edges, with vertices swapped as follows: $H \mapsto G/H$. Therefore $\Phi_n^2$ acts as the identity on cube edges. 
		
		We need to show that it acts the same way on diagonals. If $n=1$, this is clear. Now assume $n \geq 2$ and consider $D \in \D_n$. By induction, $\Phi^2$ is the identity on all diagonals, except maybe the big diagonal $(1 \to G)$.
		
		 Assume first $(1 \to G) \in D$, i.e $D \in \CC_n$. Then by (1) we get that $\Phi_n(D) \in \B_n$, and again that $\Phi_n^2(D) \in \CC_n$, so that $(1 \to G) \in D$. Now if $(1 \to G) \not \in D$, then $D \not\in \CC_n$. Then by (1) we get that $\Phi_n(D) \not\in \B_n$, and again that $\Phi_n^2(D) \not\in \CC_n$, so that $(1 \to G) \not \in D$. 
		\item Consider a diagram $D \in \Comp_d(G)$ for $$G=C_{p_1 \dots p_n}.$$ Without loss of generality we may assume that $G^0=C_{p_1 \dots p_{n-d}}$, and we know from Proposition~\ref{p:decomp} that arrows in $D$ are contained in the union of the facets $B_1,\dots,B_{n-d}$ and of $D^0=D_{G^0}^G$. We also know that $D$ contains the arrow $(G^0 \to G)$. Let $E:=\Phi_n(D)$. We claim that
		\[
			E^0=E_{G/G^0}^G
		\]
		and that the arrows of $E$ are either arrows of $E^0$ or contained in the bottom facets $B_{n-d+1},\dots,B_n$. This is equivalent to proving that $E$ contains the arrow $(G/G^0 \to G)$, and contains no arrow $(K \to L)$ with $L \supset G/G^0$ and $K \not\supset G/G^0$.
	
	Consider the induced diagram $D_1^{G^0}$. In this diagram, there are no arrows adjacent to $G^0$. Indeed, such an arrow would not be contained in $D^0$, nor in the union of facets $B_1,\dots,B_{n-d}$. Therefore, the image $E^0$ of $D_1^{G^0}$ by $\Phi_n$ must contain the big diagonal of $E^0$, namely, the arrow $$(G/G^0=C_{p_{n-d+1}\dots p_n} \to G).$$
	
	Now assume by contradiction that it contains an arrow $(K \to L)$ with $L \supset G/G^0$ and $K \not\supset G/G^0$. Let $M:=K \cap G/G^0$. As in the proof of Proposition~\ref{p:decomp}, we deduce that $E$ contains an arrow $(M \to G)$. Therefore, $D$ has no arrow to $G/M$. Now note that $G/M$ strictly contains $G^0$. Since $D$ contains the arrow $(G^0 \to G)$, the restriction condition implies that $D$ also contains the arrow $$(G^0 \cap G/M=G_0 \to G/M),$$ which is a contradiction. This concludes the proof. \qedhere
	\end{enumerate} 
\end{proof}

\begin{corollary}\label{cor:res}
Let $G=C_{p_1 \cdots p_n}$, then
$$
|\mathcal{N}_n| = 
\left\{
\begin{array}{lr}
\displaystyle{\sum_{i=0}^{n/2}}  2 \times |\operatorname{Comp}_i(G)| & n \text{ even},\\[15pt]
\displaystyle{\sum_{i=0}^{(n-1)/2}} 2 \times |\operatorname{Comp}_i(G)| + |\operatorname{Comp}_{(n+1)/2}(G)| & n \text{ odd}.
\end{array}
\right\}
$$
\end{corollary}

\section{Enumerating diagrams for $G=C_{pqr}$}\label{s:3d}

We will now use the results of the previous section to enumerate the number of $N_\infty$-operads for $G=C_{pqr}$. 
Using a code (which does not make use of any additional structure), we have calculated that there are 450 such, however, we will now show this using the theory as opposed to naive computational effort.  From Proposition~\ref{p:enum} we know that it is enough to compute the cardinalities of $\Comp_0(C_{pqr})$ and $\Comp_1(C_{pqr})$, and then $$|\mathcal{N}_3| = 2(|\Comp_0(C_{pqr})| + |\Comp_1(C_{pqr})|).$$

\begin{lemma} The size of $\Comp_3(C_{pqr})$ is 198. \label{lem-cc3} \end{lemma}

\begin{proof} 
Recall that $\Comp_3(C_{pqr})$ consists of all \Ninf-diagrams containing the arrow $(1 \rightarrow C_{pqr})$ and therefore, by restriction, all arrows $(1 \rightarrow H)$ for all $H < C_{pqr}$.
Therefore, we must count the possibilities of filling in the three two-dimensional facets containing $C_{pqr}$ in a manner that gives an \Ninf-diagram.
We distinguish the different cases according to how many ``edge'' arrows $(C_{ij} \rightarrow C_{pqr})$ there are in an \Ninf-diagram before considering the possibilities for the three top facets. We will then view the \Ninf-diagram on a top facet as an \Ninf-diagram for the two-dimensional case with top vertex $C_{pq}$, as we will also see in the figures that follow.

\begin{figure}[h]
\begin{tikzpicture}[-, node distance=2cm, auto,scale = 0.75, baseline = (qr)]
\coordinate (pqr) at (0.0,0.0);
\coordinate (pq) at (1.73,-1.0);
\coordinate (pr) at (-1.73,-1.0);
\coordinate (qr) at (0,2.0);
\coordinate (p1) at (0,-2.0);
\coordinate (q1) at (1.73,1);
\coordinate (r1) at (-1.73,1);
\node at (pqr) [above right = 0.5mm of pqr] {$pqr$}; 
\node at (p1) [below = 0.5mm of p1] {$p$};
\node at (q1) [above right = 0.5mm of q1] {$q$};
\node at (r1) [above left = 0.5mm of r1] {$r$};
\node at (pq) [below right = 0.5mm of pq] {$pq$};
\node at (pr) [below left = 0.5mm of pr] {$pr$};
\node at (qr) [above = 0.5mm of qr] {$qr$};
\draw[dashed] (p1)--(pq)--(q1)--(qr)--(r1)--(pr)--(p1);
\draw[dashed] (pqr) -- (p1);
\draw[dashed] (pqr) -- (q1);
\draw[dashed] (pqr) -- (r1);
\draw (pqr) -- (pq);
\draw (pqr) -- (pr);
\draw (pqr) -- (qr);
\end{tikzpicture}
\caption{ The three facets of the three dimensional cube containing $C_{pqr}$. The arrows $(C_{ij} \rightarrow C_{pqr})$, which are each shared by two facets, are indicated with solid lines.  }
\end{figure}
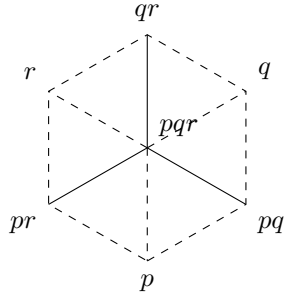


\textit{Case 1}. There are no arrows connecting to the vertex $C_{pqr}$ from a $C_{ij}$. This restricts the \Ninf-diagrams that could occur on the three top facets to those which do not contain the top arrows $(C_p \rightarrow C_{pq})$ or $(C_q \rightarrow C_{pq})$ when considered as \Ninf-diagrams for the two-dimensional case. There are five such \Ninf-diagrams (Fig.~\ref{fig-c3-hexb}.) and three faces to fill, thus $5^3 = 125$ remaining options.

\begin{figure}[h]
\begin{tikzpicture}[->, node distance=2cm, auto,scale = 0.6, baseline = (11)]
\node (11) at (0.0,2.0) {$1$};
\node (p1) at (0.0,0.0) {$p$};
\node (q1) at (2.0,2.0) {$q$};
\node (pq1) at (2.0,0.0) {$pq$};

\node (12) at (4.0,2.0) {$1$};
\node (p2) at (4.0,0.0) {$p$};
\node (q2) at (6.0,2.0) {$q$};
\node (pq2) at (6.0,0.0) {$pq$};
\draw (12) -- (p2);

\node (13) at (0.0,-2.0) {$1$};
\node (p3) at (0.0,-4.0) {$p$};
\node (q3) at (2.0,-2.0) {$q$};
\node (pq3) at (2.0,-4.0) {$pq$};
\draw (13) -- (q3);

\node (14) at (4.0,-2.0) {$1$};
\node (p4) at (4.0,-4.0) {$p$};
\node (q4) at (6.0,-2.0) {$q$};
\node (pq4) at (6.0,-4.0) {$pq$};
\draw (14) -- (p4);
\draw (14) -- (q4);

\node (15) at (8.0,-2.0) {$1$};
\node (p5) at (8.0,-4.0) {$p$};
\node (q5) at (10.0,-2.0) {$q$};
\node (pq5) at (10.0,-4.0) {$pq$};
\draw (15) -- (p5);
\draw (15) -- (q5);
\draw (15) -- (pq5);
\end{tikzpicture}
\caption{ The two dimensional \Ninf-diagrams which can occur in a facet which contains no top arrows. }
\label{fig-c3-hexb}
\end{figure}
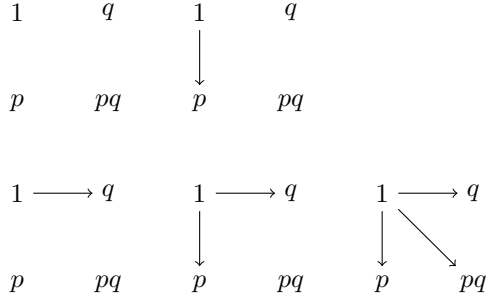

\textit{Case 2}. There is one arrow $(C_{ij} \rightarrow C_{pqr})$. Thus, one of the three top facets contains no top arrows $(C_p \rightarrow C_{pq})$ or $(C_q \rightarrow C_{pq})$ while the other two contain one arrow of those two. There are two \Ninf-diagrams of $C_{pq}$ satisfying the latter condition (Fig.~\ref{fig-c3-hexc}.), and there is a three-fold rotational symmetry, thus we have $3 \cdot 5 \cdot 2^2 = 60$ remaining options.

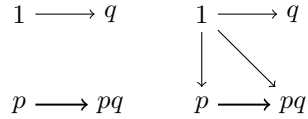
\begin{figure}[h]
\begin{tikzpicture}[->, node distance=2cm, auto,scale = 0.6, baseline = (11)]
\node (11) at (0.0,2.0) {$1$};
\node (p1) at (0.0,0.0) {$p$};
\node (q1) at (2.0,2.0) {$q$};
\node (pq1) at (2.0,0.0) {$pq$};
\draw[line width = 0.75pt] (p1) -- (pq1);
\draw (11) -- (q1);

\node (12) at (4.0,2.0) {$1$};
\node (p2) at (4.0,0.0) {$p$};
\node (q2) at (6.0,2.0) {$q$};
\node (pq2) at (6.0,0.0) {$pq$};
\draw[line width = 0.75pt] (p2) -- (pq2);
\draw (12) -- (q2);
\draw (12) -- (p2);
\draw (12) -- (pq2);

\end{tikzpicture}
\caption{ The two dimensional \Ninf-diagrams which can occur in a facet which contains one top arrow (thicker line). }
\label{fig-c3-hexc}
\end{figure}
   
\textit{Case 3}. There are two arrows of the form $(C_{ij} \rightarrow C_{pqr})$ connecting to the vertex $C_{pqr}$. Thus, one of the three top facets contains both $(C_p \rightarrow C_{pq})$ and $(C_q \rightarrow C_{pq})$ while the other two facets contain one such arrow. There is only one \Ninf-operad of $C_{pq}$ satisfying the former condition (Fig.~\ref{fig-c3-hexd}.), and there is a three-fold rotational symmetry, thus we have $3 \cdot 2^2 = 12$ remaining options.

\begin{figure}[h]
\begin{tikzpicture}[->, node distance=2cm, auto,scale = 0.6, baseline = (11)]
\node (11) at (0.0,2.0) {$1$};
\node (p1) at (0.0,0.0) {$p$};
\node (q1) at (2.0,2.0) {$q$};
\node (pq1) at (2.0,0.0) {$pq$};
\draw[line width = 0.75pt] (p1) -- (pq1);
\draw[line width = 0.75pt] (q1) -- (pq1);
\draw (11) -- (p1);
\draw (11) -- (q1);
\draw (11) -- (pq1);

\end{tikzpicture}
\caption{ The two-dimensional \Ninf-diagrams which can occur in a facet which contains both top arrows (thicker lines). }
\label{fig-c3-hexd}
\end{figure}
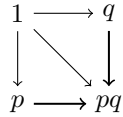

\textit{Case 4}. All three arrows connecting to the vertex $C_{pqr}$ are present. Thus all three top facets must contain two arrows $(C_p \rightarrow C_{pq})$ and $(C_q \rightarrow C_{pq})$, and so there is only one remaining option.

These cases are disjoint and account for all possible \Ninf-diagrams, and the possibilities contained therein sum to 198. \end{proof}

\begin{lemma} The size of $\Comp_1(C_{pqr})$ is 27. \label{lem-27} \end{lemma}
 
 \begin{proof}
Without loss of generality, the only arrow adjacent to $C_{pqr}$ is $(C_{pq} \rightarrow C_{pqr})$. By restriction, our \Ninf-diagram therefore also contains the parallel edges $(C_q \rightarrow C_{qr})$, $(C_p \rightarrow C_{pr})$ and $(C_1 \rightarrow C_r)$. 

The only other arrows that could occur are in the facets containing $1$ and $C_{pr}$ and $1$ and $C_{qr}$. We distinguish the possible cases according to how many of those facets contain a diagonal arrow $(C_1 \rightarrow C_{pr})$ or $(C_1 \rightarrow C_{qr})$. 

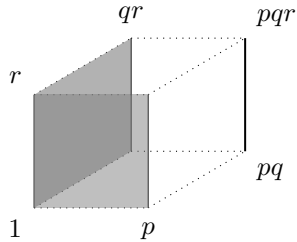
\begin{figure}[h]
\begin{tikzpicture}[-, node distance=2cm, auto,scale = 0.75, baseline = (pqr)]
\coordinate (11) at (0.0,0.0);
\coordinate (p1) at (2.0,0.0);
\coordinate (q1) at (1.7,1.0);
\coordinate (r1) at (0.0,2.0);
\coordinate (pq) at (3.7,1.0);
\coordinate (pr) at (2.0,2.0);
\coordinate (qr) at (1.7,3.0);
\coordinate (pqr) at (3.7,3.0);
\node at (11) [below left = 0.5mm of 11] {$1$}; 
\node at (p1) [below  = 0.75mm of p1] {$p$};
\node at (r1) [above left = 0.5mm of r1] {$r$};
\node at (pq) [below right = 0.5mm of pq] {$pq$};
\node at (qr) [above = 0.75mm of qr] {$qr$};
\node at (pqr) [above right = 0.5mm of pqr] {$pqr$};
\draw[dotted] (11) to (p1);
\draw[dotted] (11) to (q1); %
\draw[] (11) to (r1);
\draw[dotted] (p1) to (pq);
\draw[] (p1) to (pr);
\draw[dotted] (q1) to (pq); %
\draw[] (q1) to (qr); %
\draw[dotted] (r1) to (pr);
\draw[dotted] (r1) to (qr);
\draw[line width = 0.75] (pq) -- (pqr);
\draw[dotted] (pr) -- (pqr);
\draw[dotted] (qr) -- (pqr); 
\fill[gray, opacity=0.5] (11) -- (p1) -- (pr) -- (r1);
\fill[gray, opacity=0.6] (11) -- (r1) -- (qr) -- (q1);
\end{tikzpicture}
\caption{ The arrows induced by restriction (solid lines) if $(C_{pr}\rightarrow C_{pqr})$ is present (thicker line). An \Ninf-diagram containing $(C_{pr}\rightarrow C_{pqr})$ in $\Comp_1(C_{pqr})$can only include additional arrows in the gray facets. }
\end{figure}

\textit{Case 1}. None of the two facets contain a diagonal. There is only one such case.

\textit{Case 2}. One facet contains a diagonal, and the other one does not. The one that does not contain a diagonal can therefore contain no further arrows, whereas the other one one of the two possible diagrams in Figure \ref{figurexxx}. This therefore accounts for four possibilities. 

\textit{Case 3}. Both facets contain their diagonal. Therefore, each of them is one of the two diagrams in Figure \ref{figurexxx}. This gives us four possibilities. 

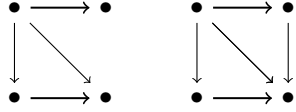
\begin{figure}[h]
\begin{tikzpicture}[->, node distance=2cm, auto,scale = 0.6, baseline = (11)]
\node (11) at (0.0,2.0) {\textbullet};
\node (p1) at (0.0,0.0) {\textbullet};
\node (q1) at (2.0,2.0) {\textbullet};
\node (pq1) at (2.0,0.0) {\textbullet};
\draw (11) -- (pq1);
\draw[line width = 0.75] (11) -- (q1);
\draw (11) -- (p1);
\draw[line width = 0.75] (p1) -- (pq1);

\node (12) at (4.0,2.0) {\textbullet};
\node (p2) at (4.0,0.0) {\textbullet};
\node (q2) at (6.0,2.0) {\textbullet};
\node (pq2) at (6.0,0.0) {\textbullet};
\draw (12) -- (pq2);
\draw[line width = 0.75] (12) -- (q2);
\draw (12) -- (p2);
\draw (12) -- (pq2);
\draw[line width = 0.75pt] (p2) -- (pq2);
\draw (q2) -- (pq2);

\end{tikzpicture}
\caption{ The possible forms of an \Ninf-diagram containing the diagonal and two parallel arrows (thicker lines). }\label{figurexxx}
\end{figure}

In all, we have nine possibilities for \Ninf-diagrams in $\Comp_1(C_{pqr})$ containing the arrow $(C_{pq} \rightarrow C_{pqr})$, so also nine for $(C_{pr} \rightarrow C_{pqr})$ and nine for $(C_{qr} \rightarrow C_{pqr})$.
 \end{proof}
 
 \begin{corollary}
 $|\mathcal{N}_3| = 198 \times 2 + 27 \times 2 = 450$.
 \end{corollary}
 
 \begin{proof}
 By Corollary \ref{cor:res}, we have $|\mathcal{N}_3|= |\Comp_3(C_{pqr})| \times 2 + |\Comp_1(C_{pqr})| \times 2$.
 \end{proof}
 
 \begin{remark}
We finish this paper by the consideration for $\mathcal{N}_n$ for $n >3$. Although we have presented a way of decomposing the problem into enumerating $\lceil n/2 \rceil$ disjoint pieces, the way forward is still not clear. Indeed, the reasoning to get the values of 198 and 27 in the $n=3$ case required studying in depth the cases appearing for $n=2$. Therefore even for $n=4$, one would have to be able to analyse the 450 options for $n=3$ on a case-by-case basis.

As such, the results appearing in this paper should be seen as a structural result as opposed to an algorithm for computation.
 \end{remark}


\bibliographystyle{amsalpha}
\bibliography{literatur}

\end{document}